\documentclass[english]{amsart}
\usepackage{babel}
\usepackage{dsfont}
\usepackage{amsmath,amsthm,amssymb,babel,amsfonts,graphics}
\usepackage{caption}
\usepackage{subcaption}
\usepackage{tikz}
\usepackage[T1]{fontenc}

\usepackage[pdftex,colorlinks]{hyperref}

\newcommand{\kom}[1]{}
\renewcommand{\kom}[1]{{\bf [#1]}}

\def\vint_#1{\mathchoice%
          {\mathop{\kern 0.2em\vrule width 0.6em height 0.69678ex depth -0.58065ex
                  \kern -0.8em \intop}\nolimits_{\kern -0.4em#1}}%
          {\mathop{\kern 0.1em\vrule width 0.5em height 0.69678ex depth -0.60387ex
                  \kern -0.6em \intop}\nolimits_{#1}}%
          {\mathop{\kern 0.1em\vrule width 0.5em height 0.69678ex
              depth -0.60387ex
                  \kern -0.6em \intop}\nolimits_{#1}}%
          {\mathop{\kern 0.1em\vrule width 0.5em height 0.69678ex depth -0.60387ex
                  \kern -0.6em \intop}\nolimits_{#1}}}
                  
                  \newcommand{\aveint}[2]{\mathchoice%
          {\mathop{\kern 0.2em\vrule width 0.6em height 0.69678ex depth -0.58065ex
                  \kern -0.8em \intop}\nolimits_{\kern -0.45em#1}^{#2}}%
          {\mathop{\kern 0.1em\vrule width 0.5em height 0.69678ex depth -0.60387ex
                  \kern -0.6em \intop}\nolimits_{#1}^{#2}}%
          {\mathop{\kern 0.1em\vrule width 0.5em height 0.69678ex depth -0.60387ex
                  \kern -0.6em \intop}\nolimits_{#1}^{#2}}%
          {\mathop{\kern 0.1em\vrule width 0.5em height 0.69678ex depth -0.60387ex
                  \kern -0.6em \intop}\nolimits_{#1}^{#2}}}

\usepackage{color}
\usepackage{epsfig}

\def\1{\raisebox{2pt}{\rm{$\chi$}}}

\newcommand{\dist}{\operatorname{dist}}

\newcommand{\kint}{\vint}

\newcommand{\eps}{\varepsilon}

\theoremstyle{plain}
\newtheorem{definition}{Definition}[section]

\newtheorem{theorem}[definition]{Theorem}
\newtheorem*{theorem*}{Main Theorem 1}
\newtheorem*{theorem**}{Main Theorem 2}

\newtheorem{lemma}[definition]{Lemma}

\newtheorem{remark}[definition]{Remark}

\theoremstyle{definition}

\theoremstyle{remark}

\numberwithin{equation}{section}

\begin{document}

\title[]{Tug-of-war games related to $p$-Laplace type equations with zeroth order terms}

\author{Jeongmin Han}
\address{Department of Mathematics and Creative Research Institute, Soongsil University, 
06978 Seoul, Republic of Korea}
\email{jeongmin.han@ssu.ac.kr}



\begin{abstract}
In this paper, we investigate a class of tug-of-war games that incorporate a constant payoff discount rate at each turn. 
The associated model problems are $p$-Laplace type partial differential equations with zeroth-order terms. 
We establish existence, uniqueness, and regularity results for the corresponding game value functions. 
Furthermore, we explore properties of the solutions to the model PDEs, informed by the analysis of the underlying games.
\end{abstract}

\maketitle
\tableofcontents

\section{Introduction}

This paper focuses on a class of tug-of-war games that incorporate a fixed discount rate at each turn.
Specifically, we study the game defined on a bounded domain $\Omega\subset\mathbb{R}^n$ associated with the following dynamic programming principle (DPP)
\begin{align}\label{dpp_damp}
u_{\eps}(x)=(1-\gamma \eps^2)\bigg\{ \frac{\alpha}{2}\bigg(\sup_{B_\eps(x)}u_{\eps}+\inf_{B_\eps(x)}u_{\eps} \bigg)
+\beta\kint_{B_\eps(x)}u_{\eps}(y)dy \bigg\}
\end{align}
satisfying $u_{\eps}\equiv F$ in the outer $\eps$-boundary strip $\Gamma_{\eps}$
for $\gamma \ge 0$ and $\alpha,\beta\in(0,1)$ with $\alpha+\beta=1$.
In addition, we also consider the corresponding PDE problem
\begin{align}\label{eqmiin}
\left\{ \begin{array}{ll}
\Delta_{p}^{N} u-(p+n)\gamma u=0 & \textrm{in $ \Omega$,}\\
 u = F & \textrm{on $ \partial \Omega$}\\
\end{array} \right.
\end{align}
with $\alpha=\frac{p-2}{p+n}$ and $\beta=\frac{n+2}{p+n}$  
based on the discussion about the DPP \eqref{dpp_damp}.
We state our main results here.
\begin{theorem*} Let $F$ be a function in $L^{\infty}(\Gamma_{\eps})$.
For $ 0< \alpha < 1 $ and $\eps>0$ with $\gamma\eps^2<\frac{1}{2}$, there exists a unique function $u_{\eps}$ satisfying \eqref{dpp_damp}. 
Moreover, $u_\eps$ satisfies
$$ |u_{\eps} (x) - u_{\eps} (z) | \le C ||F||_{L^{\infty}(\Gamma_{\eps})} \frac{ |x-z| + \eps}{r},$$ where $x, z \in B_{r}(y)$ for some $y\in \Omega$ with $B_r  (y)\subset\subset \Omega$ and $C>0$ only depends on $ \alpha, \gamma$ and $n$.
\end{theorem*}
\begin{theorem**} Let $F$ be a function in $L^{\infty}(\partial \Omega)$.
There exists a unique solution to \eqref{eqmiin} for $ 2<p< \infty$.
Moreover, $u$ satisfies
$$ |u (x) - u(z) | \le C ||F||_{L^{\infty}(\partial \Omega)} \frac{|x-z|}{r} ,$$ where $x, z \in B_{r}(y)$ for some $y\in \Omega$ with $B_r  (y)\subset\subset \Omega$ and $C>0$ depends only on $ \alpha, \gamma$ and $n$.
\end{theorem**}

Over the past decade, a variety of significant results have been established in the study of tug-of-war games. 
Much of the existing literature has focused on cases involving only terminal payoffs or running payoffs. 
However, it is also natural and mathematically interesting to consider games that incorporate a discount (or surcharge) at each turn. 
In this paper, we extend the discussion to include such games with discounting mechanisms.
The discount factor in the game setting corresponds to a zeroth-order term in the associated PDE problem.
In PDE theory, for example, it is well-known that the second-order equation
$$ a_{ij}D_{ij} u - cu =0, $$
where $a_{ij}$ satisfies $ \lambda |\xi|^2 \le a_{ij}\xi_i \xi_j \le \Lambda |\xi|^2 $ for some $0<\lambda\le \Lambda$ any $\xi \in \mathbb{R}^n$,
has the existence and uniqueness when $c\ge0$. On the other hand, there are also a number of regularity results for this equation.
A general theory for this equation can be found, for example, in \cite{MR737190,MR2777537} and especially in \cite{MR1118699,MR1351007} for fully nonlinear PDEs. For equations including lower order terms such as \eqref{eqmiin}, we refer the reader to \cite{MR1376656,MR1606359}.
We investigate a parallel theory in the context of the tug-of-war game value functions.
Specifically, the condition $\gamma \ge0$ in the DPP \eqref{dpp_damp} plays a crucial role in establishing the existence, uniqueness, and regularity of the corresponding value functions.

As a simple example, we can consider a value function of the tug-of-war with constant payoff on the boundary strip.
When $\gamma=0$, i.e., there is no discount, one can guess without difficulty that the value function of the game (or the function satisfying \eqref{dpp_damp}) would also be constant
in the domain.
However, if $\gamma >0$, the behavior of the value function would be quite different.
In that case, our intuition suggests that if the token is near the boundary, the game will be finished soon and hence the discounted amount would be relatively small.
In contrast, if the token is far from the boundary, it requires a lot of turns to finish the game and this will reduce the payoff significantly.
The following pictures show the graphs of the solutions of the problems $u_1''=0$ and $u_2''-u_2=0$ with the same boundary data $u_i(-1)=u_i(1)=1$ for $i=1,2$, which correspond to the problem \eqref{eqmiin} in $\mathbb{R}$. 
By the convergence result, Theorem \ref{conv}, we can expect that the graphs of the value functions associated with the above problems would look similar to them for sufficiently small $\eps$.
\begin{center}
\begin{figure}
    \centering
\begin{tikzpicture}[scale=2,>=stealth,font=\tiny]
  \draw[->] (-1.2,0) -- (1.2,0) node[right] {$x$}; 
  \draw[->] (0,-.4) -- (0,1.5) node[left] {$y$}; 
  \filldraw[magenta] (1,1) circle (1pt);
  \filldraw[magenta] (-1,1) circle (1pt);
  \node at (-.1,-.1) {$0$}; 
  \draw[red] plot [domain=-1:1,samples=200] (\x,1); 
    \draw[blue] plot [domain=-1:1,samples=200] (\x, {(e^(\x)+e^(-\x))/(e+1/e)});
  \foreach \x in {-1, 1}
    \draw (\x,0) node[below] {$\x$};
  \node[above left] at (0,.8) {$1$};
  \node[above, red] at (1.1,1.0) {$u_1$};
    \node[above, blue] at (1.1,0.8) {$u_2$};
\end{tikzpicture}
    \caption{The solutions of $u_1''=0$ (red) and $u_2''-u_2=0$ (blue).}
    \label{fig:enter-label}
\end{figure}
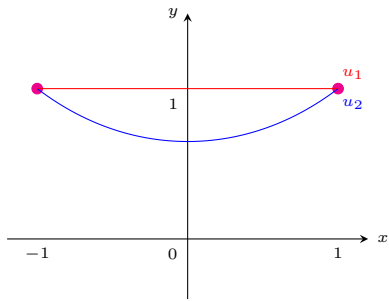
\end{center}

As we mentioned earlier, we mainly investigated \eqref{dpp_damp} and \eqref{eqmiin} throughout this paper.
To construct a stochastic game associated with \eqref{eqmiin}, we need to look at the local behavior of solutions to the problem.
Heuristic intuitions can be obtained by using a simple observation by using Taylor expansion, and we build a game based on such insights.
It is important to verify the existence and uniqueness of the value function and its relevance to the DPP \eqref{dpp_damp}.
To this end, we considered the methods in \cite{MR3161602}, which covered those issues for the case of no discounts.
On the other hand, in this paper, we presented several regularity results for the function value satisfying \eqref{dpp_damp}. We used the cancellation method to derive interior estimates for the game value, which has been used in various preceding results, for example, \cite{MR3494400,MR4125101}.
For the boundary estimate, the construction of an auxiliary game is a key step to obtain the desired result (cf. \cite{MR3011990,MR4299842}).
We also employ the notion of viscosity solution and an Arzel\`{a}-Ascoli criterion introduced in \cite{MR3011990} for the convergence of the value function.

Since the paper \cite{MR2449057} in the mid-2000s introduced tug-of-war games as an interpretation of the infinity Laplacian, there has been remarkable progress in theory on stochastic games.
Thus far, various types of tug-of-war games and the associated DPPs have been investigated, for example, \cite{MR2451291,MR2887928,MR4257616}.
On the other hand, this topic is also closely related to the $p$-Laplace type equations, and mean value characterizations of their solutions. 
In this context, one can find several papers such as \cite{MR2566554,MR2684311,MR3011990,MR3177660,MR3846232}.
In addition, it is worthwhile to mention \cite{MR3893728,MR3977220,MR4684373,MR4684385,arroyo2024krylov}, which deal with several interesting problems related to stochastic games.
Recently, there have been attempts to advance the theory concerning more general forms of dynamic programming equations in connection with this topic, such as \cite{MR4496902,MR4565418}.
We also refer the reader to \cite{MR601776,MR1199811} for a general theory of controlled processes including games.

{\bf Acknowledgments} This research was supported by the Basic Science Research Program through the National Research Foundation of Korea(NRF) funded by the Ministry of Education (RS-2021-NR060140) and by Global - Learning \& Academic research institution for Master's\textperiodcentered PhD students, and Postdocs(LAMP) Program of the National Research Foundation of Korea(NRF) grant funded by the Ministry of Education(No. RS-2025-25441317).

\section{Preliminaries}

\subsection{Heuristics}
We give a heuristic description of the relation between the DPP \eqref{dpp_damp} and the model problem \eqref{eqmiin} here.
This is closely linked to mean value characterizations of solutions to \eqref{eqmiin}.
For discussions on the mean value characterizations for the $p$-Laplace equation and its related problems, we refer to the previous results, such as \cite{MR2566554,MR2684311,MR3177660,MR4399821,teso2026game}.

Let $u \in C^{2}(\overline{\Omega})$ be a function solving the problem \eqref{dpp_damp}.
Assume that $x\in \Omega$ with $B_{\eps}(x)\subset\subset \Omega$ with small $\eps>0$. 
By using Taylor expansion, we first see that
\begin{align*}
\kint_{B_{\eps}(x)}u(y)dy= u(x)+\frac{\Delta u(x)}{2(n+2)}\eps^2 +o(\eps^2).
\end{align*}
On the other hand, if $D u\neq 0$, we can consider the following rough approximation:
\begin{align*}
\frac{1}{2}\bigg(\sup_{B_\eps(x)}u+\inf_{B_\eps(x)}u \bigg)
&\approx \frac{1}{2} \bigg\{  u\bigg(x+\eps\frac{Du(x)}{|Du(x)|}\bigg) +u\bigg(x-\eps\frac{Du(x)}{|Du(x)|}\bigg)  \bigg\}
\\& = u(x)+\frac{\langle D^2u(x)Du(x),Du(x) \rangle}{2|Du(x)|^2}\eps^2+o(\eps^2)
\\& = u(x)+\frac{\Delta_{\infty}^{N}u(x)}{2}\eps^2+o(\eps^2).
\end{align*}
Thus, for $\alpha=\frac{p-2}{p+n}$ and $\beta=\frac{n+2}{p+n}$, we see that
\begin{align*}
&(1-\gamma \eps^2)\bigg\{ \frac{\alpha}{2}\bigg(\sup_{B_\eps(x)}u+\inf_{B_\eps(x)}u \bigg)
+\beta\kint_{B_\eps(x)}u(y)dy \bigg\}
\\ & \approx (1-\gamma \eps^2)\bigg(u(x)+\frac{\Delta_{p}^{N} u(x)}{2(p+n)}\eps^2 +o(\eps^2) \bigg)
\\ & = u(x) + \frac{1}{2}\bigg( \frac{\Delta_{p}^{N} u(x)}{p+n}-\gamma u \bigg)\eps^2+o(\eps^2).
\end{align*}
Since $u$ solves
$$ \Delta_{p}^{N} u-(p+n)\gamma u=0$$
in $\Omega$, we derive
\begin{align*} 
u(x)\approx (1-\gamma \eps^2)\bigg\{ \frac{\alpha}{2}\bigg(\sup_{B_\eps(x)}u+\inf_{B_\eps(x)}u \bigg)
+\beta\kint_{B_\eps(x)}u(y)dy\bigg\}+o(\eps^2).
\end{align*}

From this observation, one can expect that $u_\eps$ satisfying \eqref{dpp_damp} has similar behaviors to $u$. 
Indeed, we will discuss the convergence issue of $u_{\eps}$ in Section \ref{sec:con}.
We also give a regularity result of $u$, Theorem \ref{pdereg}, by employing the estimates in Section \ref{sec:reg}.

\subsection{Tug-of-war game with discounts on the payoff}
\label{ssec:tow}
Now we describe a stochastic game associated with \eqref{dpp_damp}.
We construct a two-player zero-sum stochastic game.
Let $\Omega \subset \mathbb{R}^{n} $ be a bounded domain, 
$$ \Gamma_{\eps} := \{x\in \mathbb{R}^n\backslash \Omega : \dist(x,\partial \Omega)\le \eps \}$$
and $\Omega_{\eps}:=\Omega \cup \Gamma_{\eps} $. 
We begin with the game at a starting point $x_0\in\Omega$.
Each round, players decide their strategies on how to move the token; $S_{\textrm{I}}$ for Player I and $S_{\textrm{II}}$ for Player II, respectively.
Then with a probability $\alpha$, there is a fair coin toss and the winner gets the right to move the token in $B_{\eps}(x_0)$.
On the other hand, with a probability $\beta$, the token is moved randomly in $B_{\eps}(x_0)$ according to the uniform distribution.
Let the new position of the token be $x_1$. If $x_1\not\in \Omega$, Player II pays Player I the payoff $(1-\gamma\eps^2)F(x_1)$.
Otherwise, the players repeat the above procedures until the token goes outside $\Omega$
and then we can define $x_2, x_3, \dots$, and so on.
After the game is over, Player II pays Player I the payoff $(1-\gamma\eps^2)^{\tau}F(x_\tau)$, 
where $x_{\tau}\in\Gamma_{\eps}$ is the end point of the game.
We refer to \cite{MR3623556,MR4125101}, which considered the case $\gamma=0$, that is, there is no discount at each turn.

We can write this game setting more rigorously.
We first consider the history of the game defined to be a vector of the first $k+1$ game states and $k$ coin tosses.
It can be written as
$$ \big( x_0,(c_1,x_1),\dots,(c_k,x_k) \big),$$
where $c_i\in \{ \textrm{I}, \textrm{II}, 0 \}$ denotes the result of each round; $\textrm{I}$ if Player I wins, $\textrm{II}$ if Player II wins and $0$ if a random walk occurs. 
We remark that the history of the game is associated with the filtration $\{\mathcal{F}_k \}_{k=0}$, where
$ \mathcal{F}_0 := \sigma (x_0)$ and
$$ \mathcal{F}_k := \sigma \big( x_0,(c_1,x_1),\dots,(c_k,x_k) \big) \qquad \textrm{for} \quad k\ge1.$$
Now we focus on a strategy, which is a collection of Borel-measurable functions giving the next position of the token.
For $j \in \{ \textrm{I}, \textrm{II}\}$, we set 
$$S_{j}^{k}\big( x_0,(c_1,x_1),\dots,(c_k,x_k) \big)=s_{j}^{k} \in B_{1}  $$
for each $k=0,1,\dots$.
We also define sequences of random variables $ \{\xi_{k} \}_{k=0}\subset [0,1]$ and $ \{w_{k} \}_{k=0}\subset B_1$, where each $\xi_k$ and $w_k$ are randomly selected according to the uniform distribution in $[0,1]$ and $B_1$, respectively. 
Thus, we set a sequence of vector-valued random variables $\{ X_{k}^{ x_{0}}\}_{k=0}$ with
$X_{0}^{ x_{0}} \equiv x_{0}$ and 
\begin{equation}\label{gadef}X_{k}^{x_{0}} =
\left\{ \begin{array}{llll}
X_{k-1}^{x_{0}}+\eps s_{\textrm{I}}^{k-1}& \textrm{if $0\le \xi_{k-1} < \frac{\alpha}{2}$,}\\
X_{k-1}^{x_{0}}+\eps s_{\textrm{II}}^{k-1} &  \textrm{if $ \frac{\alpha}{2} \le \xi_{k-1} < \alpha$,}  \\
X_{k-1}^{x_{0}}+\eps  w_k & \textrm{if $ \alpha< \xi_{k-1} \le 1$}.\\
\end{array} \right.
\end{equation}
In this case, the stopping time $\tau$ is defined by
$$\tau= \min \{ k \ge 0 : X_{k}\in \Gamma_{\eps} \}. $$

Now the value functions for Player I and II are given by 
$$u_{\eps}^{\textrm{I}}(x_0)= \sup_{S_{\textrm{I}}}\inf_{S_{\textrm{II}}}\mathbb{E}_{S_{\textrm{I}},S_{\textrm{II}}}^{x_0}[(1-\gamma \eps^2)^{\tau} F(X_{\tau})]$$
and
$$u_{\eps}^{\textrm{II}}(x_0)= \inf_{S_{\textrm{II}}}\sup_{S_{\textrm{I}}}\mathbb{E}_{S_{\textrm{I}},S_{\textrm{II}}}^{x_0}[(1-\gamma \eps^2)^{\tau} F(X_{\tau})],$$
respectively.
We directly see that $ u_{\eps}^{\textrm{I}}\le u_{\eps}^{\textrm{II}}$ by the definition.
If those two functions coincide, we can define the game value function $u_{\eps}:=u_{\eps}^{\textrm{I}}=u_{\eps}^{\textrm{II}}$.

We note a simple observation on the value function here.
\begin{remark}\label{maxcom}
Since $\gamma>0$, we directly see that $u_{\eps}$ satisfies the following maximum principle 
$$ ||u_{\eps}||_{L^{\infty}(\Omega)}\le ||F||_{L^{\infty}(\Gamma_{\eps})}$$
by the definition.
Suppose that $u_{\eps}$ and $v_{\eps}$ satisfy the DPP \eqref{dpp_damp} with the payoff functions $F_1$ and $F_2$. respectively. Then if $F_1\le F_2 $ in $\Gamma_{\eps}$, we also have
$$ u_{\eps}\le v_{\eps} \qquad \textrm{in} \ \ \Omega_{\eps}. $$
\end{remark}

\section{The existence and uniqueness of the value function}
We consider the existence and uniqueness of our value function for the game under the setting in the previous section.
This issue was mainly dealt with for the case $\gamma=0$ in \cite{MR3161602}.
We first consider the existence and uniqueness of the function satisfying \eqref{dpp_damp},
and then we also clarify that the function indeed coincides with the game value of our tug-of-war game.
We remark that the condition $\gamma \ge 0$ is necessary to verify existence and uniqueness.
\subsection{The existence and uniqueness of the solution to  \eqref{dpp_damp}}
As we mentioned above, we investigate the existence and uniqueness of the function satisfying our DPP.

We first discuss the existence. To this end,
let $\mathcal{F}_{\eps}$ be the set of all Borel measurable and bounded functions defined on $\Omega_{\eps}$.
We define an operator $\mathcal{T}_{\eps}$ by
\begin{align}\label{deft} \mathcal{T}_{\eps}u(x) =
\left\{ \begin{array}{ll}
(1-\gamma \eps^2)\bigg\{ \frac{\alpha}{2}\big(\sup_{B_\eps(x)}u +\inf_{B_\eps(x)}u  \big)
+\beta\kint_{B_\eps(x)}u (y)dy \bigg\}& \textrm{if $ x\in \Omega$,}\\
 u(x)& \textrm{if $ x\in \Gamma_{\eps}$.}\\
\end{array} \right.
\end{align}
Then we directly see that $\mathcal{T}_{\eps}u\in \mathcal{F}_{\eps}$  for every $u\in \mathcal{F}_{\eps}$ since
$\mathcal{T}_{\eps}u$ is bounded from Remark \ref{maxcom} and $\sup_{B_{\eps}(x)}v$, $\inf_{B_{\eps}(x)}v$
are Borel in $\Omega$ for any bounded Borel function $v$ in $\Omega_{\eps}$ (see (2.3) in \cite{MR3161602}).

We use an iteration to investigate the existence of a function satisfying \eqref{dpp_damp}.
Let 
\begin{align*} u_0(x) =
 \left\{ \begin{array}{ll}
-||F||_{L^{\infty}(\Gamma_{\eps})} & \textrm{if $ x\in \Omega$,}\\
 F(x)& \textrm{if $ x\in \Gamma_{\eps}$,}\\
\end{array} \right.
\end{align*} and define $u_{k}=\mathcal{T}_{\eps}u_{k-1}$
for $k=1,2,\dots$.
We first need to show that $u_k$ converges pointwise.
For each $k\ge1$ and $x\in \Omega$, we observe that
$u_{k}(x)\le u_{k+1}(x)$ since $u_0 \le \mathcal{T}_{\eps}u_0$ and $\mathcal{T}_{\eps}$ preserves the monotonicity, i.e.,
$$ u\le v \qquad \textrm{implies} \qquad \mathcal{T}_{\eps}u\le \mathcal{T}_{\eps}v.$$
We also see that $|u_k|\le ||F||_{L^{\infty}(\Gamma_{\eps})}$ without difficulty and hence,
it follows that there exists a function $\underline{u}$ defined on $\Omega_{\eps}$ such that for each $x\in\Omega$,
$$\underline{u}(x)=\lim_{k\to \infty}u_k (x) $$
by the monotone convergence theorem. 
The uniform convergence of $\underline{u}$ can also be shown by using a similar argument in the proof of \cite[Theorem 2.1]{MR3161602}
and it follows that $\underline{u}=\mathcal{T}_{\eps}\underline{u}$.

Next, we consider the uniqueness. To do this, it is enough to show the following lemma. 
\begin{lemma}
Suppose that $u_1$ and $u_2$ satisfy the DPP \eqref{dpp_damp} with boundary data $F_1$ and $F_2$ in $\Gamma_{\eps}$, respectively.
Then, we have
$$\sup_{\Omega}|u_1-u_2|\le \sup_{\Gamma_{\eps}}|F_1-F_2|. $$
\end{lemma}
\begin{proof}
Without loss of generality, we can assume that $F_1\le F_2$ in $\Gamma_{\eps}$.
Then, we only need to show that 
\begin{align*}
M:=\sup_{\Omega}(u_2-u_1)\le \sup_{\Gamma_{\eps}}(F_2-F_1)=:m.
\end{align*} 
Suppose not. 
In that case, we have $M>m$ and this implies for any $x\in \Omega$
\begin{align}\label{ine1}\begin{split}
u_2(x)-u_1(x)& \le (1-\gamma\eps^2)\bigg\{\alpha \sup_{B_{\eps}(x)}(u_2-u_1)+\beta \kint_{B_{\eps}(x)} (u_2-u_1)(y)dy\bigg\} \\ &
\le (1-\gamma\eps^2)\bigg\{\alpha M+\beta \kint_{B_{\eps}(x)} (u_2-u_1)(y)dy\bigg\}.
\end{split}
\end{align}

Now we set
$$S:=\{x\in\Omega_{\eps}: u_2(x)-u_1(x)=M \}.$$
Since we assumed that $m<M$, we have $S\subset \Omega$.
We are going to prove that $S=\varnothing$.
By the boundedness of $\Omega$, we observe that $\overline{\Omega}$ is compact.
Thus, if $S\neq\varnothing$, we can take a sequence $\{x_k\}\subset S\cap\Omega$ such that
\begin{align*}
\lim_{k\to\infty} (u_2-u_1)(x_k)=M \qquad \textrm{and} \qquad  \lim_{k\to\infty} x_k=x_0\in \overline{\Omega}.
\end{align*}
Then we can deduce
$$ \kint_{B_{\eps}(x_0)}(u_2-u_1)(y)dy=\lim_{k\to\infty}\kint_{B_{\eps}(x_k)}(u_2-u_1)(y)dy$$
by the absolute continuity of the Lebesgue integral.
On the other hand, it also follows that
\begin{align*}
M&=\lim_{k\to\infty} (u_2-u_1)(x_k)\\ & \le
(1-\gamma\eps^2)\bigg\{\alpha M+\beta \lim_{k\to\infty}\kint_{B_{\eps}(x_k)}(u_2-u_1)(y)dy \bigg\}
\\&= (1-\gamma\eps^2)\bigg\{\alpha M+\beta \kint_{B_{\eps}(x_0)}(u_2-u_1)(y)dy \bigg\}
\\ & \le (1-\gamma\eps^2)M =M-\gamma\eps^2 M
\end{align*}
by \eqref{ine1}.
However, we have already assumed that $\gamma>0$ and $$M>m=\sup_{\Gamma_{\eps}}(F_2-F_1)\ge0,$$
and this is a contradiction.
Therefore, we have $M\le m$ and we can finish the proof.
\end{proof}
We can directly observe the uniqueness of the solution to \eqref{dpp_damp} by applying the above lemma with $F_1=F_2$.

\subsection{The existence of the game value}
We have observed the existence and uniqueness of the function satisfying \eqref{dpp_damp} in the previous subsection.
It remains to verify that the function indeed coincides with the value function of the game introduced in Section \ref{ssec:tow}.
We can show this by constructing a sub(or super)martingale (see also \cite[Theorem 3.2]{MR3161602}).

\begin{theorem}Let $u_{\eps}$ be the function satisfying the DPP \eqref{dpp_damp} with boundary data $F$ and
$u_{\eps}^{\textrm{I}}$ and $u_{\eps}^{\textrm{II}}$ are the corresponding game values for Player I and Player II, respectively.
Then, we have $u_{\eps}=u_{\eps}^{\textrm{I}}=u_{\eps}^{\textrm{II}}$ in $\Omega_{\eps}$.
\end{theorem}
\begin{proof}
We will give the proof of
$$u_{\eps}^{\textrm{II}}\le u_{\eps} \qquad \textrm{in} \  \Omega $$
and then this also gives us $ u_{\eps}^{\textrm{I}}\ge u_{\eps}$ in $\Omega$ by symmetry.
Since we already know that $u_{\eps}=u_{\eps}^{\textrm{I}}=u_{\eps}^{\textrm{II}}=F$ in $\Gamma_{\eps}$, we can obtain our desired result in that case. 

We consider the following situation: Each turn, Player II takes a strategy which almost minimizes $u_{\eps}$. 
More precisely, for some fixed $\eta>0$ and given $x_k\in \Omega$, Player II takes a point $x_{k+1}\in B_{\eps}(x_k)$ such that
$$u_{\eps}(x_{k+1})\le  \inf_{B_{\eps}(x_k)} u_{\eps} +\eta 2^{-k}.$$
We denote by $S_{\textrm{II}}^0$ this strategy. 
Then, we have for any strategy $S_{\textrm{I}}$,
\begin{align*}
&\mathbb{E}_{S_{\textrm{I}},S_{\textrm{II}}^0}^{x_0}[(1-\gamma\eps^2)^{k+1}u_{\eps}(x_{k+1})+\eta 2^{-k}|\mathcal{F}_k]
\\ &\le (1-\gamma\eps^2)^{k+1}\bigg\{ \frac{\alpha}{2}\bigg( \inf_{B_{\eps}(x_k)}u_{\eps}+\eta 2^{-k} 
+ \sup_{B_{\eps}(x_k)}u_{\eps}\bigg)+\beta\kint_{B_{\eps}(x_k)}u_{\eps}(y)dy\bigg\}+\eta 2^{-k}
\\ & \le (1-\gamma\eps^2)^{k}u_{\eps}(x_k)+\eta 2^{-(k-1)},
\end{align*} 
and this yields that $M_k=(1-\gamma\eps^2)^{k} u_{\eps}(x_k)+\eta 2^{-(k-1)}$ is a supermartingale.
Since $F$ is bounded, we can observe that $M_k$ is also bounded. 
Therefore, by the optional stopping theorem, we have
\begin{align*}
u_{\textrm{II}}^{\eps}(x_0)&=\inf_{S_{\textrm{II}}}\sup_{S_{\textrm{I}}}\mathbb{E}_{S_{\textrm{I}},S_{\textrm{II}}}^{x_0}[(1-\gamma \eps^2)^{\tau} F(X_{\tau})]
\\ & \le \sup_{S_{\textrm{I}}}\mathbb{E}_{S_{\textrm{I}},S_{\textrm{II}}^0}^{x_0}[ (1-\gamma\eps^2)^{\tau}F(X_{\tau}) + \eta 2^{-\tau}]
\\ & \le \sup_{S_{\textrm{I}}}\mathbb{E}_{S_{\textrm{I}},S_{\textrm{II}}^0}^{x_0}[ u(x_0) + 2\eta ]
\\ & = u(x_0) + 2\eta ,
\end{align*}
and hence, we can finish the proof.
\end{proof}

\section{Regularity estimates}
\label{sec:reg}

Regularity is one of the important research subjects in PDE theory; It gives us some information about the behavior of solutions.
Strictly, the DPP problems are different from PDE problems, nevertheless, we can still discuss the regularity of solutions.
There have been a number of regularity results for game values of tug-of-war games, for example, \cite{MR3494400,MR3698169,MR3846232,MR4125101}. 
In this section, we investigate the regularity issue for \eqref{dpp_damp}.
In particular, we present interior and boundary regularity estimates of our game value function $u_{\eps}$ satisfying \eqref{dpp_damp}.
This is important in itself and also essential for discussing the convergence of the value function.

We first state and prove our interior regularity result. Our method is based on the coupling method, which has been used in previous studies, for example, \cite{MR3846232,MR3494400}.
\begin{theorem}\label{intes}Assume that $B_{2}(0) \subset \Omega$. 
Suppose that $u_{\eps}$ satisfies \eqref{dpp_damp} for $ 0< \alpha < 1 $ and $\eps>0$ with $\gamma\eps^2<\frac{1}{2}$.
Then we have
$$ |u_{\eps} (x) - u_{\eps} (z) | \le C ||u_{\eps}||_{L^{\infty}(B_2)} ( |x-z| + \eps) ,$$ where $x, z \in B_{1}(0)$ and $C>0$ only depends on $ \alpha, \gamma$ and $n$.
\end{theorem}
\begin{proof}
We first set an appropriate auxiliary function (see also Section 2.2 in \cite{MR4125101}). 
Define an increasing function $\omega:[0,\infty)\to[0,\infty)$ satisfying
$$\omega(t) = t- \omega_{0}t^{\frac{3}{2}}  \qquad \textrm{for} \quad 0 \le t \le \omega_{1} := (3 \omega_{0})^{-2},$$ where $ \omega_{0} $ is a fixed constant with $\omega_0>1/2$. 
Observe that $$ \omega '(t) = 1-\frac{3}{2}\omega_{0}t^{\frac{1}{2}} \in \bigg[ \frac{1}{2},1 \bigg] \qquad \textrm{for} \ \ 0 \le t \le \omega_{1}$$ and$$ \omega ''(t) = -\frac{3}{4}\omega_{0}t^{-\frac{1}{2}} <0 \qquad \textrm{for} \ \ 0 \le t \le \omega_{1}.$$
Then we see that $\omega$ is increasing, strictly concave and $C^{2}$ in $(0, \omega_1]$.
Let $f_1$ and $f_2$ be as follows:
\[
\begin{split}
f_1(x,z)=C\omega(|x-z|)+|x+z|^2\,
\end{split}
\]
and
\begin{equation*} f_{2}(x, z) = \left\{ \begin{array}{ll}
C^{2(N-i)} \eps & \textrm{if $(x, z) \in A_{\textrm{I}}$,}\\
0 & \textrm{if $|x-z|>N \eps / 10 $,}
\end{array} \right.
\end{equation*}
where $C>1$ and
$$A_{\textrm{I}} = \{ (x , z) \in \mathbb R ^ {2n} : (i-1) \eps / 10 < |x-z| \le i \eps / 10 \} $$ for $ i = 0, 1, ... , N$, where $N$ is a sufficiently large number satisfying $N>100C$.  
Then we consider $f=f_1-f_2$ as an auxiliary function.

Now we assume that $$\sup_{B_2\times B_2} (u_{\eps}(x)-u_{\eps}(z)) \le 1$$
by using a normalization.
Let $T:=\{(x,x)  : x\in \Omega\} $ be the diagonal set of $\Omega$.
By using $f_1\geq 1$ in $(B_{2}\times B_{2}) \backslash (B_{1}\times B_{1})$,
we see that 
\begin{equation} \label{as2}
\begin{split}
u_{\eps}(x)-u_{\eps}(z)-f(x,z)
&= u_{\eps}(x)-u_{\eps}(z)-f_1(x,z)+f_2(x,z)\\ 
&\leq \max f_2
\le C^{2N}\eps 
\end{split}
\end{equation}
for $(x,z) \in (B_{2}\times B_{2}) \backslash ((B_{1}\times B_{1})\backslash T)$.
Since we can assume $x=-z$ with an appropriate translation, it is sufficient to show that
\begin{align}\label{wts_o}
\sup_{(x,z) \in  B_{2}\times B_{2}  } (u_{\eps}(x)-u_{\eps}(z)-f(x,z))\le C^{2N}\eps 
\end{align}
to derive our desired regularity result.

Suppose not. Then we have 
\begin{align} \label{as}
M := \sup_{(x,z) \in  (B_{1}\times B_{1}) \backslash T } (u_{\eps}(x)-u_{\eps}(z)-f(x,z)) > C^{2N}\eps .
\end{align}
We can obtain \eqref{wts_o} if we show that \eqref{as} contradicts. 
Observe that
\begin{align*}
M= \sup_{(x,z) \in  B_{2}\times B_{2}  } (u_{\eps}(x)-u_{\eps}(z)-f(x,z)) 
\end{align*}
in this case.
Set $g(x,z)=u_{\eps}(x)-u_{\eps}(z)$.
Then for small enough $\eta >0$, we can select $(x_{1},z_{1}) \in  (B_{1}\times B_{1}) \backslash T  $ such that
\begin{align*}
M \le u_{\eps}(x_{1})-u_{\eps}(z_{1})-f(x_{1},z_{1})+\eta=g(x_{1},z_{1})-f(x_{1},z_{1})+\eta  .
\end{align*}

For $x\in \mathbb{R}^n$, $\nu\in B_1$ and a measurable function $\psi$, we define
$$\mathcal{A}_{\eps}\psi(x,\nu)= \alpha \psi (x+\eps \nu)+\beta \kint_{B_1}\psi(x+\eps y) dy.$$
We see that
\begin{align*}
\mathcal{A}_{\eps}u_{\eps}(x,\nu)-\mathcal{A}_{\eps}u_{\eps}(z,\mu)
=\alpha g(x+\eps\nu,z+\eps\mu)+\beta \kint_{B_1}g(x+\eps y , z+\eps y)dy.
\end{align*}
If we set a functional $G_{\eps}$ such that
$$G_{\eps}\Psi(x,z,\nu,\mu)=\alpha \Psi (x+\eps \nu, z+\eps \mu)+ \beta\kint_{B_1}\Psi(x+\eps y,z+\eps y)dy,$$
we have 
$$\mathcal{A}_{\eps}u_{\eps}(x,\nu)-\mathcal{A}_{\eps}u_{\eps}(z,\mu) = G_{\eps}g(x,z,\nu,\mu).$$
We can also check without difficulty that $\Psi_1\le\Psi_2$ implies $G_{\eps}\Psi_1\le G_{\eps}\Psi_2 $.
Since we have assumed
$g\le f+M$ in $B_2\times B_2$, we get
\begin{align*}
G_{\eps}g(x_{1},z_{1},\nu,\mu)\le M+G_{\eps}f(x_{1},z_{1},\nu,\mu) .
\end{align*} 

Observe that
\begin{align}\label{2ndpp}\begin{split}
g(x,z)=(1-\gamma\eps^2)\bigg\{ \frac{\alpha}{2}\bigg(&\sup_{B_\eps (x)\times B_{\eps}(z)}g+\inf_{B_\eps (x)\times B_{\eps}(z)}g \bigg) \\&
+\beta\kint_{B_1}g(x+\eps h, z+\eps h)dy \bigg\}.
\end{split}
\end{align}
It can be represented by
\begin{align}\label{2ndpp2}
g(x,z)=\frac{1-\gamma\eps^2}{2}\bigg(\sup_{\substack{|\nu|\le1\\|\mu|\le1}}G_{\eps}g(x,z,\nu,\mu)
+\inf_{\substack{|\nu|\le1\\|\mu|\le1}}G_{\eps}g(x,z,\nu,\mu)\bigg).
\end{align}
Then, we can derive
\begin{align*}
M&\le g(x_{1},z_{1})-f(x_{1},z_{1})+\eta
\\ & = \frac{1-\gamma\eps^2}{2} \bigg(\sup_{\substack{|\nu|\le1\\|\mu|\le1}}G_{\eps}g(x_1,z_1,\nu,\mu)
+\inf_{\substack{|\nu|\le1\\|\mu|\le1}}G_{\eps}g(x_1,z_1,\nu,\mu)\bigg)-f(x_{1},z_{1})+\eta
\\ & \le \frac{1-\gamma\eps^2}{2} \bigg(\sup_{\substack{|\nu|\le1\\|\mu|\le1}}G_{\eps}f(x_1,z_1,\nu,\mu)
+\inf_{\substack{|\nu|\le1\\|\mu|\le1}}G_{\eps}f(x_1,z_1,\nu,\mu)\bigg)-f(x_{1},z_{1})
\\ & \qquad \qquad +(1-\gamma\eps^2)M+\eta
\end{align*}
and this is rearranged as
\begin{align*}
\gamma \eps^2 M \le \frac{1-\gamma\eps^2}{2} \bigg(\sup_{\substack{|\nu|\le1\\|\mu|\le1}}G_{\eps}f(x_1,z_1,\nu,\mu)
+\inf_{\substack{|\nu|\le1\\|\mu|\le1}}G_{\eps}f(x_1,z_1,\nu,\mu)\bigg)-f(x_{1},z_{1})+\eta.
\end{align*}
Hence, we deduce the contradiction if we show
\begin{align}\label{wtp}
(1-\gamma\eps^2) \bigg(\sup_{\substack{|\nu|\le1\\|\mu|\le1}}G_{\eps}f(x,z,\nu,\mu)
+\inf_{\substack{|\nu|\le1\\|\mu|\le1}}G_{\eps}f(x,z,\nu,\mu)\bigg)-2f(x,z)<2\gamma \eps^2 M
\end{align}
for any $(x,z)\in B_2 \times B_2$.
Since $M>0$, it is enough to prove that
\begin{align}\label{wtp_alt}
(1-\gamma\eps^2) \bigg(\sup_{\substack{|\nu|\le1\\|\mu|\le1}}G_{\eps}f(x,z,\nu,\mu)
+\inf_{\substack{|\nu|\le1\\|\mu|\le1}}G_{\eps}f(x,z,\nu,\mu)\bigg)-2f(x,z)<0.
\end{align}

We first assume that $|x-z|> N\eps/10$.
In that case, we have $ f_2=0$ and this implies $f=f_1$.
Hence, we can rewrite \eqref{wtp_alt} by
\begin{align}\label{wtp_alt2}
(1-\gamma\eps^2) \bigg(\sup_{\substack{|\nu|\le1\\|\mu|\le1}}G_{\eps}f_1(x,z,\nu,\mu)
+\inf_{\substack{|\nu|\le1\\|\mu|\le1}}G_{\eps}f_1(x,z,\nu,\mu)\bigg)-2f_1(x,z)<0.
\end{align}
In the proof of \cite[Lemma 5.2]{MR4125101}, one can find that $f_1$ satisfies
\begin{align}
\sup_{\substack{|\nu|\le1\\|\mu|\le1}}G_{\eps}f_1(x,z,\nu,\mu)
+\inf_{\substack{|\nu|\le1\\|\mu|\le1}}G_{\eps}f_1(x,z,\nu,\mu)-2f_1(x,z)<0.
\end{align}
Then it follows that 
\begin{align*}
 & (1-\gamma\eps^2) \bigg(\sup_{\substack{|\nu|\le1\\|\mu|\le1}}G_{\eps}f_1(x,z,\nu,\mu)
+\inf_{\substack{|\nu|\le1\\|\mu|\le1}}G_{\eps}f_1(x,z,\nu,\mu)\bigg)-2f(x,z)  
\\ & < -2\gamma\eps^2f(x,z)
\\ & <0
\end{align*}
since $f_1(x,z)>0$ when $x\neq z$, and this yields \eqref{wtp_alt2}.

Next, we consider the other case.
If $|x-z|\le N\eps/10$, 
we first observe that 
$$u(x)-u(x)-f(x,x)= C^{2N}\eps -4|x|^{2} <M$$
and thus \eqref{as} cannot occur when $x=z$. 
Hence, we only need to show \eqref{wtp_alt} when $0<|x-z| \le \frac{N}{10} \eps$.
Since
\begin{align*}
\big||a| -|b| \big| \le |a-b| 
\end{align*}
for any $a,b \in \mathbb{R}^{n}$, 
it follows that
\begin{align*}
\big||(x+\eps v)-(z+ \eps w)| -|x-z| \big| \le \eps |v-w| .
\end{align*}
Now we get
\begin{align*}
|f_{1}((x,z)+\eps (v,w))-f_{1}(x,z)| & \le 
C\eps |v-w| + 2\eps |\langle x+z,v+w \rangle|+\eps^{2}|v+w|^{2}
\\ & \le 2C\eps + 8\eps +4\eps^{2}
\\ & \le 3C\eps 
\end{align*}
for any unit vectors $v$ and $w$, and sufficiently large $C$. 
Then we obtain
\begin{align}\label{estf1}\begin{split}
&(1-\gamma\eps^2) \sup_{\substack{|\nu|\le1\\|\mu|\le1}}G_{\eps}f_{1}(x,z,\nu,\mu)-f_{1}(x,z) \\ & = (1-\gamma\eps^2)\bigg( \sup_{\substack{|\nu|\le1\\|\mu|\le1}}G_{\eps}f_{1}(x,z,\nu,\mu)-f_{1}(x,z) \bigg)+\gamma \eps^2 f_1(x,z)
\\ &\le (1-\gamma\eps^2)3C\eps +\gamma\eps^2 (2C+4)
\\ & \le 3(1+\gamma)C\eps .\end{split}
\end{align}
On the other hand, we see that
$$\sup_{\substack{|\nu|\le1\\|\mu|\le1}}G_{\eps}f(x,z,\nu,\mu) \le \sup_{\substack{|\nu|\le1\\|\mu|\le1}}G_{\eps}f_1(x,z,\nu,\mu)$$
since $f_{2}\ge 0$ and $f=f_{1}-f_{2}$.
We also get
$$  \inf_{\substack{|\nu|\le1\\|\mu|\le1}}G_{\eps}f(x,z,\nu,\mu) \le  \sup_{\substack{|\nu|\le1\\|\mu|\le1}}G_{\eps}f_1 (x,z,\nu,\mu)- \sup_{\substack{|\nu|\le1\\|\mu|\le1}}G_{\eps}f_2 (x,z,\nu,\mu).$$
Recall that we have assumed that $|x-z|<N\eps/10$.
By the definition of the set $A_i$, We observe that $(x,z) \in A_{\textrm{I}}$ for some $i=1,\cdots,N$.
Then we can select $\nu, \mu $ with $|\nu|=|\mu|=1 $ such that $(x,z)+\eps(\nu,\mu) \in A_{i-1}$.
This implies 
\begin{align*}
 \sup_{\substack{|\nu|\le1\\|\mu|\le1}}G_{\eps}f_2 (x,z,\nu,\mu)& \ge \alpha f_{2}((x,z)+\eps(\nu, \mu))
\\ & = \alpha C^{2(N-i+1)}\eps 
\\ & = \alpha C^{2(N-i)}\eps \bigg(C^{2}-\frac{2}{\alpha( 1-\gamma\eps^2)}\bigg) + \frac{2}{1-\gamma\eps^2}f_{2}(x,z).
\end{align*}
Choose $C$ large enough such that
$$ C^{2}-\frac{4}{\alpha }>12(1+\gamma)C.$$
In this case, we see that $C$ satisfies
$$ C^{2}-\frac{2}{\alpha (1-\gamma\eps^2)}>6\frac{1+\gamma}{1-\gamma\eps^2}C,$$
and hence, we obtain 
\begin{align} \label{supf2} \sup_{\substack{|\nu|\le1\\|\mu|\le1}}G_{\eps}f_2 (x,z,\nu,\mu)>  6C\frac{1+\gamma}{1-\gamma\eps^2}\eps  +\frac{2 f_{2}(x,z)}{1-\gamma\eps^2}.
\end{align}
Now we observe that 
\begin{align}\label{estf2}\begin{split}  & (1-\gamma\eps^2)\inf_{\substack{|\nu|\le1\\|\mu|\le1}}G_{\eps}f(x,z,\nu,\mu)\\ &
\le  (1-\gamma\eps^2)\bigg(\sup_{\substack{|\nu|\le1\\|\mu|\le1}}G_{\eps}f_1 (x,z,\nu,\mu)- \sup_{\substack{|\nu|\le1\\|\mu|\le1}}G_{\eps}f_2 (x,z,\nu,\mu) \bigg)
\\ &  < (1-\gamma\eps^2)\bigg(  \sup_{\substack{|\nu|\le1\\|\mu|\le1}}G_{\eps}f_1(x,z,\nu,\mu) -
6C\frac{1+\gamma}{1-\gamma\eps^2}\eps  -\frac{2 f_{2}(x,z)}{1-\gamma\eps^2} \bigg)
\\ & \le (f_{1}(x,z)+3C(1+\gamma)\eps )  -
6C(1+\gamma)\eps  -  2f_{2}(x,z)
\\ & =f_1(x,z)-2f_2(x,z)-3C(1+\gamma)\eps  .\end{split}
\end{align}
Combining \eqref{estf1} and \eqref{estf2}, we finally obtain \eqref{wtp_alt}
and the proof is completed.
\end{proof}

It is also important to observe how the value function $u_{\eps}$ behaves near the boundary. 
We derive a boundary regularity result for \eqref{eqmiin} here.

To discuss boundary regularity, we first need to introduce several assumptions.
We first assume that $\Omega$ satisfies an exterior sphere condition, that is,
for any $y \in \partial \Omega$, there exists $ B_{\delta}(z) \subset \mathbb{R}^{n} \backslash \Omega $ with $\delta > 0 $ such that $y \in  \partial B_{\delta}(z)$.
Moreover, we also assume that the payoff function $F$ satisfies
\begin{align} \label{bdlip}
|F(x)-F(z)| \le L_F|x-z|
\end{align}
for some $L_F>0$ and any $x,z\in \Gamma_{\eps}$.

In \cite{MR3011990}, the key step to the proof of the boundary regularity result was to estimate the stopping time $\tau$. 
To this end, the authors considered an auxiliary stochastic process with a longer stopping time than the original game.
We utilize a similar approach to obtain our desired result.
Let us construct a stochastic process as follows:
Let $y \in \partial \Omega$ and take $z \in \mathbb{R}^{n} \backslash \Omega $ with $ B_{\delta}(z) \subset \mathbb{R}^{n} \backslash \Omega $ and $y \in \partial B_{\delta}(z)$. 
And we assume that $\Omega \subset B_{R}(z)$ for some large $R>0$.
We consider the game in $B_R(z)\backslash \overline{B}_{\delta}(z)$, and fix the strategy of Player I to pull toward $z$ (we denote the strategy by $S_{\textrm{I}}^z{}$). 
We set the token cannot escape outside $\overline{B}_{R}(z)$ in this process,
and hence, it ends only if the token is located in $\overline{B}_{\delta}(z)$. 
We denote by $$ \tau^{\ast} = \inf \{ k : x_{k} \in \overline{B}_{\delta}(z)  \}.$$
Then one can obtain the following estimate of $\tau^{\ast}$ by using a similar argument to \cite[Lemma 4.5]{MR3011990}
and \cite[Lemma 5.2]{MR4299842}.  
\begin{lemma} \label{estauxsp} 
Under the setting above, we have
$$ \mathbb{E}^{x_{0}}[\tau^{\ast}] \le    \frac{C(n, \alpha,  R/\delta)( \dist(\partial B_{\delta}(y),x_{0})+o(1))}{\eps^{2}}   $$
for any $ x_{0}  \in \Omega  \subset  B_{R}(z)\backslash \overline{B}_{\delta}(z) $.
Here $o(1) \to 0$ as $\eps \to 0$.
\end{lemma} 

By using Lemma \ref{estauxsp}, we can derive a boundary estimate for the value function.
\begin{theorem}\label{bdes}
Assume that $\Omega$ satisfies the exterior sphere condition and $F$ satisfies \eqref{bdlip}.
Then for the value function $u_{\eps}$ with boundary data $F$, we have
\begin{align}\label{nearbd} 
& |u_{\eps}(x)-u_{\eps}(z)|  \le C(n,  R/\delta, \gamma)||F||_{C^{0,1}(\Gamma_{\eps})}( |x_0-y|+o(1)) +||F||_{C^{0,1}(\Gamma_{\eps})}\delta
\end{align} 
for any $x\in\Omega$ and $z\in \Gamma_{\eps}$.
\end{theorem}
\begin{proof}Since the boundary data $F$ satisfies \eqref{bdlip}, it is sufficient to estimate 
$$ \sup_{S_{\textrm{I}}}\inf_{S_{\textrm{II}}}\mathbb{E}_{S_{\textrm{I}},S_{\textrm{II}}}^{x_0}[|x_{\tau}-z|].$$

In the proof of \cite[Lemma 4.6]{MR3011990}, we can find the following estimates: 
$$\mathbb{E}_{S_{\textrm{I}}^{z},S_{\textrm{II}}}^{x_0}[|x_{\tau}-z|]\le  |x_0-z|+C\eps \mathbb{E}_{S_{\textrm{I}}^{z},S_{\textrm{II}}}^{x_0}[\tau]$$
for some universal constant $C>0$ and 
$$\eps^2 \mathbb{E}_{S_{\textrm{I}}^{z},S_{\textrm{II}}}^{x_0}[\tau]\le \eps^2\mathbb{E}_{S_{\textrm{I}}^{z},S_{\textrm{II}}}^{x_0}[\tau^{\ast}]
\le C(n,  R/\delta)( \dist(\partial B_{\delta}(y),x_{0})+o(1)). $$
We also have $y \in B_{\delta}(z)$ and this implies
$$ \mathbb{E}_{S_{\textrm{I}}^{z},S_{\textrm{II}}}^{x_0}[|x_{\tau}-z|]\le C(n,  R/\delta)( |x_0-y|+o(1)) +\delta.$$
Now we observe that
\begin{align*}
\mathbb{E}_{S_{\textrm{I}}^{z},S_{\textrm{II}}}^{x_0}\big[F(x_{\tau})-&\big|F(x_{\tau})\big(1-(1-\gamma \eps^2)^{\tau}\big)\big|\big]
\\&\le
\mathbb{E}_{S_{\textrm{I}}^{z},S_{\textrm{II}}}^{x_0}[(1-\gamma \eps^2)^{\tau}F(x_{\tau})]
\\& \le \mathbb{E}_{S_{\textrm{I}}^{z},S_{\textrm{II}}}^{x_0}\big[F(x_{\tau})+\big|F(x_{\tau})\big(1-(1-\gamma \eps^2)^{\tau}\big)\big|\big]
\end{align*}
and
\begin{align*}
 \mathbb{E}_{S_{\textrm{I}}^{z},S_{\textrm{II}}}^{x_0}\big[\big|F(x_{\tau})\big(1-&(1-\gamma \eps^2)^{\tau}\big)\big|\big]
 \\ & \le  \mathbb{E}_{S_{\textrm{I}}^{z},S_{\textrm{II}}}^{x_0}\big[\big|F(x_{\tau})\big(1-(1-\gamma \tau \eps^2)\big)\big|\big]
 \\&\le  \gamma||F||_{L^{\infty}(\Gamma_{\eps})} \mathbb{E}_{S_{\textrm{I}}^{z},S_{\textrm{II}}}^{x_0}[\tau \eps^2]
 \\&\le  \gamma||F||_{L^{\infty}(\Gamma_{\eps})}C(n,  R/\delta)(  |x_0-y|+o(1)).
\end{align*}
On the other hand, we can also deduce that
\begin{align*}
F(z)-&C(n,  R/\delta, L_F)( |x_0-y|+o(1)) -L_F\delta
\\& \le \mathbb{E}_{S_{\textrm{I}}^{z},S_{\textrm{II}}}^{x_0}[F(x_{\tau})] \le 
F(z)+C(n,  R/\delta, L_F)( |x_0-y|+o(1)) +L_F\delta
\end{align*}
since we have assumed \eqref{bdlip}.
Therefore, it follows that
\begin{align*}
&\sup_{S_{\textrm{I}}}\inf_{S_{\textrm{II}}}\mathbb{E}_{S_{\textrm{I}},S_{\textrm{II}}}^{x_0}\big[(1-\gamma \eps^2)^{\tau}F(x_{\tau})\big]
\\&\ge \inf_{S_{\textrm{II}}}\mathbb{E}_{S_{\textrm{I}}^{z},S_{\textrm{II}}}^{x_0}\big[(1-\gamma \eps^2)^{\tau}F(x_{\tau})\big]
\\ & \ge F(z)-C(n,  R/\delta, \gamma)||F||_{C^{0,1}(\Gamma_{\eps})}( |x_0-y|+o(1)) -||F||_{C^{0,1}(\Gamma_{\eps})}\delta .
\end{align*}
Similarly, we can also obtain 
\begin{align*}
&\sup_{S_{\textrm{I}}}\inf_{S_{\textrm{II}}}\mathbb{E}_{S_{\textrm{I}},S_{\textrm{II}}}^{x_0}\big[(1-\gamma \eps^2)^{\tau}F(x_{\tau})\big]
\\ & \le F(z)+C(n,  R/\delta, \gamma)||F||_{C^{0,1}(\Gamma_{\eps})}( |x_0-y|+o(1)) +||F||_{C^{0,1}(\Gamma_{\eps})}\delta .
\end{align*}
by fixing the strategy of Player II.
Recalling
$$u_{\eps}(x_0)= \sup_{S_{\textrm{I}}}\inf_{S_{\textrm{II}}}\mathbb{E}_{S_{\textrm{I}},S_{\textrm{II}}}^{x_0}\big[(1-\gamma \eps^2)^{\tau}F(x_{\tau})\big] $$
and $u_{\eps}(z)=F(z)$ for $z\in \Gamma_{\eps}$,
we finally obtain that for $x \in \Omega$ and $z \in \Gamma_{\eps}$,
\begin{align*}
 |u_{\eps}(x_0)-u_{\eps}(z)|   \le C(n,  R/\delta, \gamma)||F||_{C^{0,1}(\Gamma_{\eps})}( |x_0-y|+o(1)) +||F||_{C^{0,1}(\Gamma_{\eps})}\delta .
\end{align*}
\end{proof}

\section{Convergence result for \eqref{eqmiin}}
\label{sec:con}
The convergence of value functions is also an important issue for studies on stochastic games.
In this section, we investigate the convergence of our value function as the step size $\eps$ goes to zero.
We will show this by using the notion of viscosity solutions (cf. \cite{MR3011990,MR2875296,MR3441079}, etc).
Meanwhile, such a convergence result can give us a corresponding regularity result for the model problem \eqref{eqmiin}.
To this end, we also need to verify the uniqueness of the solution to \eqref{eqmiin}.

We begin this section with the definition of a viscosity solution.
\begin{definition} \label{vissol} 
A function $u \in C(\Omega) $ is a viscosity solution to \eqref{eqmiin} if
the following conditions hold:
\begin{itemize}
\item[(a)] for all $ \varphi \in C^{2}(\Omega ) $ touching $u$ from above at $x_{0} \in \Omega $, 
\begin{align*} 
\left\{ \begin{array}{ll}
 \Delta_{p}^{N}\varphi(x_{0})  \ge (p+n)\gamma \varphi(x_{0}) \qquad \qquad \textrm{if $ D\varphi (x_{0}) \neq 0 $,}\\
\lambda_{\max}((p-2)D^{2}\varphi(x_{0}) )  + \Delta\varphi(x_{0}) \ge (p+n)\gamma \varphi(x_{0})  \qquad \textrm{if $D\varphi (x_{0}) = 0   $.}\\
\end{array} \right. 
\end{align*}
\item[(b)] for all $ \varphi \in C^{2}(\Omega ) $ touching $u$ from below at $x_{0} \in \Omega $, 
\begin{align*} 
\left\{ \begin{array}{ll}
 \Delta_{p}^{N}\varphi(x_{0})  \le (p+n)\gamma \varphi(x_{0}) \qquad \qquad \textrm{if $ D\varphi (x_{0}) \neq 0 $,}\\
\lambda_{\min}((p-2)D^{2}\varphi(x_{0}) )  + \Delta\varphi(x_{0}) \le (p+n)\gamma \varphi(x_{0})  \qquad \textrm{if $D\varphi (x_{0}) = 0   $.}\\
\end{array} \right. 
\end{align*}
\end{itemize}
Here, the notations $\lambda_{\max}(X)$ and $\lambda_{\min}(X)$ mean the largest and the smallest eigenvalues of a symmetric matrix $X$. 
\end{definition}

The following Arzel\`{a}-Ascoli criterion is essential to show the convergence of $u_{\eps}$.
One can find the proof of this lemma in \cite[Lemma 4.2]{MR3011990}.
\begin{lemma} \label{aras}
Let $ \{ u_{\eps} : \overline{\Omega} \to \mathbb{R}, \eps>0 \} $ be a set of functions such that
\begin{itemize}
\item[(a)] there exists a constant $C >0$ so that $|u_{\eps}(x)| < C $ for every $ \eps > 0$ and every $x \in \overline{\Omega} $. 
\\ \item[(b)] given $ \eta > 0$, there are constants $r_{0}$ and $\eps_{0}$ so that for every $\eps >0$ and $x,z \in \overline{\Omega}$ with $|x-z| < r_{0} $, it holds
$$ |u_{\eps}(x)-u_{\eps}(z)| < \eta .$$
\end{itemize}
Then, there exists a uniformly continuous function $u: \overline{\Omega} \to \mathbb{R} $ and a subsequence $\{ u_{\eps_{\textrm{I}}} \} $ such that $ u_{\eps_{\textrm{I}}}  $ uniformly converges to $u$  in $\overline{\Omega}$, as $i \to \infty$.
\end{lemma}

By combining the above lemma and the regularity results in the previous section (Theorem \ref{intes} and Theorem \ref{bdes}), 
we can prove the convergence of the value function under the assumption of Theorem \ref{bdes}.
\begin{theorem}\label{conv}Assume that $\Omega$ satisfies the exterior sphere condition and  $ F \in C^1(\Gamma_{\eps}) $
satisfies \eqref{bdlip}.
Let $u_{\eps}$ denote the solution to \eqref{dpp_damp} with boundary data $ F  $ for each $\eps>0$.
Then, there exist a function $ u: \overline{\Omega}_{\eps} \to \mathbb{R}$ and a subsequence 
$ \{ \eps_{\textrm{I}} \} $ such that 
$$ u_{\eps_{\textrm{I}}} \to u \qquad \textrm{uniformly in} \quad \overline{\Omega}.$$
Furthermore, $u$ is a viscosity solution to \eqref{eqmiin}.
\end{theorem}
\begin{proof}
We first observe the existence of a uniform convergent subsequence of $\{u_{\eps}\}$.
By using the definition of $u_{\eps} $, we have
 $$||u_{\eps} ||_{L^{\infty}(\Omega)} \le ||F||_{L^{\infty}(\Omega)} < \infty $$
for any $\eps > 0$ and thus $u_{\eps} $ are uniformly bounded. 
We also see the equicontinuity in the sense of Lemma \ref{aras} by Theorem \ref{intes} and Theorem \ref{bdes},
and hence we can apply the Arzel\`{a}-Ascoli criterion.
Therefore, we can find a uniformly convergent subsequence, still denoted by $u_{\eps}$, to a function $u \in C(\overline{\Omega})$.

Next, we verify that $u$ is a viscosity solution to \eqref{eqmiin}.
We first observe that
$$ u(x) = \lim_{\eps \to 0}u_{\eps}(x) = F(x)$$
for any $x \in \partial\Omega$.
Thus, it is sufficient to show that $u$ is a solution to 
$$\Delta_{p}^{N} u-(p+n)\gamma u=0   \qquad \textrm{in} \ \Omega$$
in the viscosity sense.
Without loss of generality, it is enough to prove that $u$ satisfies (b) in Definition \ref{vissol}. 

We fix $x \in \Omega$ and consider a small neighborhood of $x$ such that 
$ B_{R}(x)  \subset \Omega $ for some $R>0$.
For each $\eps>0$, we can consider a function $\varphi \in C^{2}(B_{R}(x))$ touching $u$ from below at $x$.
Since $u_{\eps}$ converges uniformly to $u$, for sufficiently small $\eps >0$, we can find a convergent sequence $\{x_{\eps}\}$
satisfying the following property 
$$ (u_{\eps}-\varphi)(z) \ge (u_{\eps}-\varphi)(x_{\eps})-\eta_{\eps} \qquad\textrm{for any} \ \ z\in B_{R}(x) $$
for some $x_{\eps}\in B_{R}(x)$ and small $\eta_{\eps}>0$. 
We observe that $x_{\eps} \to x$ as $ \eps \to 0$ in that case.

Recall \eqref{deft}. For $x \in \Omega$,
we have
\begin{align*}u_{\eps}(x)=
\mathcal{T}_{\eps}u_{\eps}(x) =(1-\gamma \eps^2)\bigg\{ \frac{\alpha}{2}\bigg(\sup_{B_\eps(x)}u_{\eps} +\inf_{B_\eps(x)}u_{\eps}  \bigg)
+\beta\kint_{B_\eps(x)}u_{\eps} (y)dy \bigg\}.
\end{align*}
Set $ \psi = \varphi +( u_{\eps}-\varphi)(x_{\eps})$.
Then we observe that
$u_{\eps} \ge \psi - \eta_{\eps}$ in $B_{R}(x) $.
By using this, it follows that
\begin{align*}
u_{\eps}(x_{\eps})=\mathcal{T}_{\eps}u_{\eps}(x_{\eps}) 
 \ge \mathcal{T}_{\eps}\big(\psi(x_{\eps})-\eta_{\eps}) =\mathcal{T}_{\eps}\psi(x_{\eps})-
 (1-\gamma \eps^2)\eta_{\eps}
\end{align*}
and 
\begin{align*}
 \mathcal{T}_{\eps}\psi(x_{\eps}) =\mathcal{T}_{\eps}\varphi(x_{\eps})+(1-\gamma \eps^2)( u_{\eps}-\varphi)(x_{\eps}).
\end{align*}
Hence, we get  
\begin{align*}
u_{\eps}(x_{\eps}) \ge \mathcal{T}_{\eps}\varphi(x_{\eps})+(1-\gamma \eps^2)\big(( u_{\eps}-\varphi)(x_{\eps})-\eta_{\eps}\big)
\end{align*}
and this yields 
\begin{align} \label{2est1} \gamma\eps^2  (u_{\eps}-\varphi)(x_{\eps})+(1-\gamma \eps^2)\eta_{\eps} \ge \mathcal{T}_{\eps}\varphi (x_{\eps})-\varphi (x_{\eps}). 
\end{align}

By the Taylor expansion, we can compute that 
\begin{align*} (1-\gamma \eps^2)&\bigg\{ \frac{\alpha}{2}\bigg(\sup_{B_\eps(x)}\varphi  +\inf_{B_\eps(x)}\varphi   \bigg)
+\beta\kint_{B_\eps(x)}\varphi  (y)dy \bigg\}
\\&= \varphi(x) + \frac{1}{2}\bigg( \frac{\Delta_{p}^{N} \varphi(x)}{p+n}-\gamma \varphi(x) \bigg)\eps^2+o(\eps^2).
\end{align*}
Since we have assumed that $\varphi$ has the $C^2$-regularity, $\varphi$ attains its local minimum at a point $  x_{\eps}^1 $ in $\overline{B}_\eps(x)$, that is,
$$\varphi(x_{\eps}^1)=\inf_{B_{\eps}(x)}\varphi .$$
Assume that $x_{\eps}^1 \neq x$. Then for $\tilde{x}_{\eps}^1=2x-x_{\eps}^1$ (the mirror point of $x_{\eps}^1$ with respect to $x$), we have 
\begin{align*}
& (1-\gamma \eps^2)\bigg\{ \frac{\alpha}{2}\bigg(\sup_{B_\eps(x)}\varphi  +\inf_{B_\eps(x)}\varphi   \bigg)
+\beta\kint_{B_\eps(x)}\varphi  (y)dy \bigg\} - \varphi (x) \\
 & \ge  (1-\gamma \eps^2)\bigg\{ \frac{\alpha}{2}\big(\varphi(x_{\eps}^1)  +  \varphi(\tilde{x}_{\eps}^1)  \big)
+\beta\kint_{B_\eps(x)}\varphi  (y)dy \bigg\} - \varphi (x) 
\\ & = \frac{1}{2} \eps^{2} \bigg(  \frac{(p-2) \langle D^{2} \varphi (x) \nu_{\eps}, \nu_{\eps} \rangle+\Delta\varphi (x)  }{p+n} - \gamma  \varphi (x)\bigg)+o(\eps^2),
\end{align*}
where $\nu_{\eps}=(x-x_{\eps}^1)/|x-x_{\eps}^1|$. 
By \eqref{2est1}, we obtain 
\begin{align} \begin{split} \label{2est2}
 \gamma\eps^2  (u_{\eps}-\varphi)(x_{\eps})&+(1-\gamma \eps^2)\eta_{\eps} \\&\ge   
 \frac{1}{2} \eps^{2} \bigg(  \frac{(p-2) \langle D^{2} \varphi (x) \nu_{\eps}, \nu_{\eps} \rangle+\Delta\varphi (x)  }{p+n} - \gamma \varphi (x)\bigg)+o(\eps^2)  .
\end{split}
\end{align}

Suppose that $D\varphi (x) \neq 0 $. Since $x_{\eps} \to x$ when $\eps \to 0$, it follows that
$$\nu_{\eps}  \to -\frac{D\varphi (x)}{|D\varphi (x)|} $$
as $\eps \to 0$ (we remark that we can find a proper subsequence of $\{\nu_{\eps}\}$ in that case due to $D\varphi (x) \neq 0 $).
and this implies 
$$(p-2) \langle D^{2} \varphi (x) \nu_{\eps}, \nu_{\eps} \rangle+\Delta\varphi (x) \to \Delta_{p}^{N} \varphi (x)   .$$
Meanwhile, we also have
$$(u_{\eps}-\varphi)(x_{\eps}) \to (u-\varphi)(x) =0$$
by the uniform convergence of $u_{\eps}$. 
We take $\eta=o(\eps^2)$ and
divide both side in \eqref{2est2} by $\eps^{2}$.
By letting $\eps \to 0$, we finally get
$$ 0 \ge \Delta_{p}^{N} \varphi (x) -(p+n)\gamma \varphi (x)  .$$

On the other hand, if $D\varphi (x) = 0  $,
we observe that 
\begin{align*} 
& (1-\gamma \eps^2)\bigg\{ \frac{\alpha}{2}\bigg(\sup_{B_\eps(x)}\varphi  +\inf_{B_\eps(x)}\varphi   \bigg)
+\beta\kint_{B_\eps(x)}\varphi  (y)dy \bigg\} - \varphi (x) \\
 & \ge  (1-\gamma \eps^2)\bigg\{ \frac{\alpha}{2}\big(\varphi(x_{\eps}^1)  +  \varphi(\tilde{x}_{\eps}^1)  \big)
+\beta\kint_{B_\eps(x)}\varphi  (y)dy \bigg\} - \varphi (x) 
\\ & \ge \frac{1}{2} \eps^{2} \bigg(  \frac{(p-2) \lambda_{\min}(D^{2}\varphi(x))+\Delta\varphi (x)  }{p+n} - \gamma  \varphi (x)\bigg)+o(\eps^2).
\end{align*}
By using $x_{\eps} \to x$ as $\eps \to 0$ and the continuity of the map $z  \mapsto \lambda_{\min}(D^{2}\varphi(z)) $, we can verify that
\begin{align} \label{2est3} 0 \ge \frac{1}{p+n} \big\{ \Delta\varphi(x) + (p-2)\lambda_{\min}(D^{2}\varphi(x) ) \big\}
-\gamma \varphi (x)
\end{align}
by a similar computation in the previous case.

We can also show the inequality in the opposite direction by considering a function $\varphi $ touching $u$ from above
and using a similar argument.
Then we complete the proof.
\end{proof}

For the uniqueness, we can find the relevant result for \eqref{eqmiin} in \cite[Appendix D]{MR3698169}.
Combining the regularity estimates, the convergence for \eqref{dpp_damp} and the uniqueness result, we get the following regularity result of $u$ solving \eqref{eqmiin}.
\begin{theorem}\label{pdereg}Let $u$ be the solution of \eqref{eqmiin} for $ 2<p< \infty$.
Then we have
$$ |u (x) - u(z) | \le C ||F||_{L^{\infty}(\partial \Omega)} \frac{|x-z|}{r} ,$$ where $x, z \in B_{r}(y)$ for some $y\in \Omega$ with $B_r  (y)\subset\subset \Omega$ and $C>0$ depends only on $ \alpha, \gamma$ and $n$.
\end{theorem}


\end{document}